\documentclass{amsart}
\usepackage{amssymb}
\usepackage{amsfonts}
\usepackage{amsthm}
\usepackage{amsmath}
\usepackage{setspace}
\usepackage{enumitem}

\newtheorem{theorem}{Theorem}[section]
\newtheorem{lemma}[theorem]{Lemma}

\newtheorem{example}[theorem]{Example}
\newtheorem{corollary}[theorem]{Corollary}
\newtheorem{proposition}[theorem]{Proposition}
\newtheorem{remark}[theorem]{Remark}
\newtheorem{definition}[theorem]{Definition}
\newtheorem{definitions}[theorem]{Definitions}

\newcommand{\N}{\mathbb{N}}

\begin{document}
\title{\textbf{{More about cofinally complete metric spaces}}}

\author{Lipsy}
\address{Department of Mathematics, Indian Institute of Technology Delhi, New Delhi-110016, India}
\email{lipsy1247@gmail.com}

\author{Manisha Aggarwal}
\address{Department of Mathematics, St. Stephen's College, Delhi 110007, India}
\email{manishaaggarwal.iitd@gmail.com}

\author{S. Kundu}
\address{Department of Mathematics, Indian Institute of Technology Delhi, New Delhi-110016, India}
\email{skundu@maths.iitd.ac.in}
\thanks{}

\begin{abstract}
Metric spaces satisfying properties stronger than completeness
and weaker than compactness have been studied by many authors over the
years. One such significant family is that of cofinally complete metric spaces. We discuss the relationship between cofinally complete metric spaces and the family of almost uniformly continuous functions, which has recently been introduced by Kyriakos Keremedis in \cite{[K]}. We also discuss several equivalent conditions for metric spaces whose completions are cofinally complete in terms of some geometric functional.
\end{abstract}



\maketitle

\section{Introduction}
The concepts of completeness and compactness play an  inevitable role in the theory of metric spaces. Some classes of metric spaces satisfying properties stronger than completeness and weaker than compactness have been the subject of study for a number of articles over the years. One such well known metric space is Atsuji space, also widely known as UC space. A metric space $(X, d)$ is called an Atsuji space if every real-valued continuous function on $X$ is uniformly continuous. Probably J. Nagata was the first one to study such spaces in 1950 in \cite{[JN]}, while in 1958 several new equivalent characterizations of such spaces were studied by M. Atsuji in \cite{[MA]}. Moreover, a wide collection of equivalent conditions for a metric space to be an Atsuji space can be found in the survey article \cite{[KJ1]} by Kundu and Jain. \\
\indent For discussing completeness of a metric space, one has to consider its corresponding Cauchy sequences. Recall that a sequence $(x_n)$ in a metric space $(X, d)$ is said to be Cauchy if for each $\epsilon >0$, there exists a residual set of indices $\mathbb{N}_\epsilon$ such that each pair of terms whose indices come from $\mathbb{N}_\epsilon$ are within $\epsilon$ distance of each other, and $(X, d)$ is called complete if each Cauchy sequence in $X$ clusters. In 1971, Howes \cite{[H2]} generalized the notion of Cauchy sequences in terms of nets, which he called cofinally Cauchy; replacing "residual" by "cofinal" in the definition, we obtain sequences that are called cofinally Cauchy. More precisely, a sequence $(x_n)$ in a metric space $(X,d)$ is said to be cofinally Cauchy if for every $\epsilon>0$, there exists an infinite subset $\mathbb{N}_\epsilon$ of $\mathbb{N}$ such that for each n, j $\in \mathbb{N}_\epsilon$, we have $d(x_n,x_j)<\epsilon$. Much later in 2008, Beer \cite{[B1]} cast new light on those metric spaces in which each cofinally Cauchy sequence has a cluster point and called those metric spaces to be cofinally complete metric spaces. Evidently, the cofinally complete metric spaces sit between the complete metric spaces and the Atsuji spaces. Motivating from the significance of Cauchy-continuous functions, recently in 2016, Aggarwal and Kundu \cite{[AK1]} defined another class of functions called cofinally Cauchy regular (or CC-regular for short). A function $f: (X,d) \rightarrow (Y,\rho)$ between two metric spaces $(X, d)$ and $(Y, \rho)$ is called Cauchy-continuous (also referred as Cauchy-sequentially regular; CS-regular for short) if for any Cauchy sequence $(x_n)$ in $(X,d)$, ($f(x_{n}))$ is Cauchy in $(Y,\rho)$, while $f$ is called cofinally Cauchy regular (or CC-regular) if $f(x_n)$ is cofinally Cauchy in $(Y,\rho)$ for every cofinally Cauchy sequence $(x_n)$ in $(X,d)$. In \cite{[AK2]}, the authors gave several equivalent characterizations of a cofinally complete metric space in terms of CC-regular functions.\\
\indent Functions play a significant role in the theory of metric spaces. Two important classes of functions, namely the class of continuous functions and that of uniformly continuous functions, are well known to all of us. A function $f: (X,d) \rightarrow (Y,\rho)$ between two metric spaces $(X, d)$ and $(Y, \rho)$ is called uniformly continuous if $\forall~\epsilon >0,~\exists$ a $\delta >$ 0 such that for every $A \subseteq X$, if the diameter of $A$ is $< \delta$ then $f(A) \subseteq B(b, \epsilon)$ for some $b \in$ $Y$. Recently in 2017, Keremedis \cite{[K]} generalized the defintion of uniformly continous functions and defined a new class of functions which he called almost bounded functions. If we replace one open ball of radius $\epsilon$ in the metric space $Y$ by union of finite open balls of radius $\epsilon$ in the definition of uniformly continuous functions, we get the definition of almost bounded functions. More precisely, A function $f$ from a metric space $(X,d)$ to a metric space $(Y, \rho)$  is called \textit{almost bounded} if for $\forall~\epsilon > $ 0 $\exists$ a $\delta >$ 0 such that for every $A \subseteq X$, if the diameter of $A$ is $< \delta$ then there is a finite subset $B = \left\{b_{1}, b_{2}, ...,b_{n}\right\}$ of $Y$ such that $f(A) \subseteq \displaystyle\bigcup_{i=1}^{n}B(b_{i}, \epsilon)$. It is well known that the class of uniformly continous functions is contained in the class of continous functions. Unfortunately, almost bounded functions need not be continuous. Hence Keremedis defined those functions which are continuous as well as almost bounded and called them almost uniformly continous functions. He also cast light on those metric spaces on which every continuous function is almost uniformly continuous. Similar to the definition of UC spaces, he called such metric spaces almost uniformly continuous spaces or AUC spaces for short. The class of AUC spaces lies strictly in between the class of Atsuji spaces and the class of complete metric spaces. This has inspired us to find relationship between cofinally complete metric spaces and AUC spaces. The second section of this paper studies the relation between cofinally complete metric sapces and AUC spaces through some functions which are used to characterize the aforesaid metric spaces. Meanwhile, we obtain new characterizations of totally bounded metric spaces.\\
\indent In the third section, we explore some interesting properties of almost uniformly continuous spaces. We exploit some generalizations of asymptotic sequences defined in \cite{[AK1]} to obtain sequential criteria of almost uniformly continuous functions. Recently in \cite{[AK2]}, authors have found several interesting classes of bounded real-valued functions on metric spaces and subsequent characterization of such metric spaces. This has inspired us to characterize those metric spaces on which every almost uniformly continuous function is bounded. Interestingly such metric spaces are equivalent to totally bounded metric spaces. We also obtain a new characterization of cofinally complete metric spaces and Atsuji spaces. \\
\indent In section 4, we briefly study those maps which preserve cofinally complete subspaces of a metric spaces $(X,d)$. In the last section, we give some characterizations for a metric space to have cofinal completion.
\\
\indent The symbols $\mathbb{R}$, $\mathbb{N}$ and $\mathbb{Q}$ denote the sets of real numbers, natural numbers	and rational numbers respectively. Unless mentioned otherwise, $\mathbb{R}$ and its subsets
carry the usual distance metric and all metric spaces are infinite. If $(X, d)$ is a
metric space, $x \in X$ and $\delta > 0$, then $B(x, \delta)$ denotes the open ball in $(X, d)$, centered at $x$ with radius $\delta$. Also, $(\widehat{X}, d)$ denotes the completion of a metric space $(X,d)$.
\medskip
\section{Relation Between Cofinally Complete Metric Spaces and AUC Spaces}

 In 1971, Howes \cite{[H2]} introduced a new class of sequences, obtained by replacing residual with cofinal in the definition of Cauchy sequences and called them cofinally Cauchy sequences. In 2008, Beer 
 \cite{[B1]} gave various characterizations of the metric spaces in which every cofinally Cauchy sequence clusters, knows as cofinally complete metric spaces. Let us state their precise definitions.
\begin{definition}
	A sequence $(x_{n})$ in a metric space $(X,d)$ is called cofinally Cauchy if $\forall \epsilon > 0$, there exists an infinite subset $\mathbb{N}_\epsilon$ of $\mathbb{N}$ such that for each $n,j \in \mathbb{N}_\epsilon$, we have $d(x_n,x_j) < \epsilon.$\\
	\indent A metric space $(X, d)$ is said to be cofinally complete if every cofinally Cauchy sequence in $X$ clusters, while $(X, d)$ is called an Atsuji space if every real-valued continuous function on $X$ is uniformly continuous.
\end{definition}
The class of Cauchy continuous functions is a well-known example of a class that lies strictly between the class of uniformly continuous functions and that of continuous functions. Recently in \cite{[AK1]}, another class of functions has been introduced that preserves cofinally Cauchy sequences.
\begin{definition}
	A function $f: (X,d) \rightarrow (Y,\rho)$ between two metric spaces is said to be Cauchy-continuous if ($f(x_{n}))$ is Cauchy in $(Y,\rho)$ for every Cauchy sequence $(x_{n})$ in $(X,d)$ while $f$ is called cofinally Cauchy regular (or CC-regular) if $f(x_n)$ is cofinally Cauchy in $(Y,\rho)$ for every cofinally Cauchy sequence $(x_n)$ in $(X,d)$.
\end{definition}
\noindent In \cite{[AK2]}, the authors gave a useful equivalent characterization of cofinally complete metric spaces: 	A metric space $(X, d)$ is cofinally complete if and only if each continuous function from $(X, d)$ to any metric space $(Y, \rho)$ is CC-regular.
\noindent A metric space $(X,d)$ is called an Atsuji space if every real-valued continuous function on $(X,d)$ is uniformly continuous. Since every Atsuji space is cofinally complete, one may think of having a similar characterization of cofinally complete metric spaces. In fact, in \cite{[B1]} Beer provided a nice equivalent condition in this regard: A metric space $(X, d)$ is cofinally complete if and only if each continuous function defined on it is uniformly locally bounded. Let us recall the definition of uniformly locally bounded functions.

\begin{definition}
	A function $g: (X,d)\rightarrow (Y,\rho)$ between two metric spaces is called \textit{uniformly locally bounded}\index{uniformly locally bounded} if $\exists$ $ \delta>0$ such that $\forall$ $ x\in X,$ $ g(B_d(x,\delta))$ is a bounded subset of $(Y, \rho)$.
\end{definition}

\begin{proposition}\label{P1}
	A metric space $(X, d)$ is cofinally complete if and only if each continuous function defined on it is uniformly locally bounded.
\end{proposition}

\noindent  Recall that a function $f: (X,d) \rightarrow (Y,\rho)$ between two metric spaces $(X, d)$ and $(Y, \rho)$ is called uniformly continous if $\forall~\epsilon >0,~\exists$ a $\delta >$ 0 such that for every $A \subseteq X$, if the diameter of $A$ is $< \delta$ then $f(A) \subseteq B(b, \epsilon)$ for some $b \in$ $Y$. Recently in 2017, Keremedis \cite{[K]} generalized the notion of uniformly continuous functions and defined a new class of functions which he called almost bounded functions, but almost bounded functions need not be continuous, so he defined another class of functions which are continuos and almost bounded, and called them almost uniformly continuous functions. Let us give the precise definitions.

\begin{definition}
	A function $f$ from a metric space $(X,d)$ to a metric space $(Y, \rho)$  is called \textit{almost bounded} if for $\forall~\epsilon > $ 0 $\exists$ a $\delta >$ 0 such that for every $A \subseteq X$, if the diameter of $A$ is $< \delta$ then there is a finite subset $B = \left\{b_{1}, b_{2}, ...,b_{n}\right\}$ of $Y$ such that $f(A) \subseteq \displaystyle\bigcup_{i=1}^{n}B(b_{i}, \epsilon)$.
\end{definition}
\begin{definition}
	A continuous almost bounded function from a metric space $(X,d)$ to a metric space $(Y, \rho)$ is called \textit{almost uniformly continuous}.
\end{definition}

Similar to the definition of Atsuji spaces (UC spaces), the author in \cite{[K]} defined AUC spaces as follows.
\begin{definition}
	A metric space $(X,d)$ is said to be an \textit{almost uniformly continuous or AUC} if every continuous real valued function on $X$ is almost uniformly continuous.
\end{definition}

\begin{remark}
	Every uniformly continous function is almost uniformly continuous, but an almost uniformly continuous function need not be even Cauchy continuous. For example, let $X= \{\frac{1}{n}: n \in \N \}$ with the usual metric and let $f:X \rightarrow \mathbb{R}$ be defined as:\\
	$$f\bigg(\frac{1}{n}\large\bigg)=\left\{ \begin{array}{lll}
	1     & : & n ~is ~odd\\
	2     & :&  n ~is ~even
	\end{array}\right\}$$
\end{remark}
\noindent Now we state relations between the above defined functions used to characterize AUC spaces and cofinally complete spaces. Consequently, we get relationship between AUC spaces and cofinally complete spaces.
\begin{proposition}
	Let $f$ be a an almost bounded function from a metric space $(X,d)$ to a metric space $(Y, \rho)$, then $f$ is uniformly locally bounded.
\end{proposition}
\begin{proof}
	Let $f$ be an almost bounded function from a metric space $(X,d)$ to a metric space $(Y, \rho)$. Thus $\forall~\epsilon > $ 0 $\exists$ a $\delta >$ 0 such that  $\forall~{x}\in X,~ \exists$ finite subset $\left\{y_{1}, y_{2}, ...,y_{n}\right\}$ of $Y$ such that
	$f(B_{x}(x, \delta)) \subseteq \displaystyle\bigcup_{i=1}^{n}B(y_{i}, \epsilon).$
	So we get a $\delta$ such that $\forall~x \in X$, $f(B_{x}(x, \delta))$ is a bounded subset, which implies $f$ is uniformly locally bounded. 
\end{proof}

\begin{remark}
	A uniformly locally bounded need not be almost bounded, e.g., let $X=\{\frac{1}{n} : n \in \mathbb{N}\}$ with the usual topology and let $Y=\{n : n \in \mathbb{N}\}$ with the metric: $d(x,y) = 1$ if $x \neq y$, otherwise 0. Define $f$ from $X$ to $Y$ such that $f(\frac{1}{n}) = n~\forall n \in \mathbb{N}$.
\end{remark}

\begin{proposition}
	Every almost uniformly continuous function from from a metric space $(X,d)$ to a metric space $(Y, \rho)$ is \textit{CC-regular}.
\end{proposition}
\begin{proof}
	Let $(x_{n})$ be any \textit{cofinally cauchy sequence} in $X$ having no constant sequence. For every $\epsilon > $ 0 there exists a $\delta > $ 0 such that $\forall$ ${A}\subseteq{X}$ such that $d(A) < \delta,$ $\exists~ \left\{y_{1}, y_{2}, ...,y_{n}\right\} \subseteq {Y}$ such that $f(A) \subset \displaystyle\bigcup_{i=1}^{n}B(y_{i}, \epsilon)$.
	For this $\delta$, there exists infinite subset $\mathbb{N}^{'}$ of $\mathbb{N}$ such that $d(x_{n}, x_{m}) < \delta$ $\forall~{n,m} \in \mathbb{N}^{'}$. Let $A = \left\{x_{n}: n \in \mathbb{N}^{'} \right\}$.
	Thus $f(A) \subset \displaystyle\bigcup_{i=1}^{n}B(y_{i}, \epsilon)$ for some $y_{i}\in Y$.
	At least one of these $B(y_{i}, \epsilon)$ will contain elements of the type $f(x_{n})$ for infinite values of $n$. Hence the sequence $(f(x_{n}))$ is cofinally cauchy.
\end{proof}
\begin{corollary}
	If a metric space $(X,d)$ is an almost Atsuji space, then $(X,d)$ is cofinally complete.
\end{corollary}
\begin{proof}
	Since a metric space $(X,d)$ is cofinally complete if and only if every continuous function is \textit{CC-regular}, the statement is true.
\end{proof}

\noindent	Our next result says that the converse of the above corollary is also true.
\begin{proposition}
	Every \textit{cofinally complete} metric space $(X,d)$ is Almost Atsuji.
\end{proposition}
\begin{proof}Let $f: X \rightarrow \mathbb{R}$ be any continuous function. Since $X$ is cofinally complete, $f$ is uniformly locally bounded. Thus, $\exists~\delta > 0$ such that $f(B(x, \delta))$ is bounded $\forall x \in X $. Let $\epsilon > 0$ and let $x \in X$, since $f(B(x,\delta))$ is bounded, therefore $\exists~M > $ 0 such that $\lvert f(y) \rvert <  M ~ \forall y \in B(x, \delta)$.
	Since every bounded subset of $\mathbb{R}$ is totally bounded, we can find $y_{1}, y_{2},...y_{n} \in \mathbb{R}$ such that $f(B(x,\delta)) \subset \displaystyle\bigcup_{i=1}^{n}B(y_{i}, \epsilon)$. Thus $f$ is almost uniformly continuous.
\end{proof}
\noindent Thus we conclude that the class of AUC spaces and the class of cofinally complete metric spaces is same.
\noindent We have seen that every almost uniformly continuous function is CC-regular, our next results gives relationship between CC-regular functions and almost bounded functions.
\begin{theorem}\label{T1}
	Let $f$ be a function between two metric spaces $(X, d)$ and $(Y, \rho)$, $f$ is CC-regular if and only if $f$ is almost bounded.
\end{theorem}
\begin{proof}
	Let $f\left(X,d\right) \rightarrow \left(Y,\rho\right)$ be a \textit{CC-regular} function.
	Suppose $f$ is not almost bounded. Therefore $\exists \epsilon > 0$ such that $\forall n \in \mathbb{N}$, $\exists x_n \in X$ such that $f(B(x_n, \frac{1}{n}))$ can not be written as finite union of open balls of radius $\epsilon$ in $Y$. Hence there exits a sequence $\left(f\left(x_n^m\right)\right)_{m\in\mathbb{N}}$ in $f \left(B\left(x_n, \frac{1}{n}\right)\right)$ such that  $\rho\left(f\left(x_n^m\right),f\left(x_n^t\right)\right) \geq \epsilon~\forall ~m,~t \in \mathbb{N}$. According to the hypothesis the sequence $\left(x_n^m\right)_{m\in\mathbb{N}}$ is not cofinally Cauchy for all $ n \in \mathbb{N}$. Now for a fixed n, the balls $B\left(f\left(x_n^m\right),\epsilon \right)$ are disjoint $\forall m$. Let $A_n$ = $\left\{ f(x^m_n) : m \in \mathbb{N} \right\}$. Let $F_1=\left\{f(x^1_1)\right\}$. Clearly, we can choose $f(x^2_2)$ (rename if necessary) $\in A_2$ such that $\rho(f(x^1_1),f(x^2_2)) \geq \frac{\epsilon}{2}$. Now, choose $f(x^2_1) \in A_1\backslash \left\{f(x^1_1),f(x^2_2)\right\}$ such that $F_2=\left\{f(x^1_1),f(x^2_2),f(x^2_1)\right\}$ is $\frac{\epsilon}{2}$ discrete. Suppose finite subset $F_n$ of $\displaystyle\bigcup_{i=1}^{n} A_i$ is chosen satisfying: 1) $F_n$ is $\frac{\epsilon}{2}$ discrete, 2) $|F_n \cap A_i|$ = $n-i+1~\forall~1\leq i\leq n$. Now choose $f(x^{n+1}_{n+1})$ $\in A_{n+1}$ such that $F_n \cup \left\{f(x^{n+1}_{n+1})\right\}$ is $\frac{\epsilon}{2}$ discrete. Suppose it is not possible, then $\rho(y, F_n)\leq \frac{\epsilon}{2} ~\forall~ y \in A_{n+1}.$ Therefore, there exists $y, ~y' \in A_{n+1}\backslash F_n,~ y\neq y'$ such that $\rho(y,z) < \frac{\epsilon}{2}$ and $\rho(y',z) < \frac{\epsilon}{2}$ for some $z \in F_n$ which implies $\rho(y,y') < \epsilon$, which is a contradiction. Thus, $F_{n+1} = F_n \bigcup \left\{f(x^{n+1}_1), f(x^{n+1}_2),...f(x^{n+1}_{n_+1})\right\}$ is $\frac{\epsilon}{2}$ discrete. Now if we take the sequence in the order we chose the elements, the sequence is $\frac{\epsilon}{2}$ discrete but its pre-image is cofinally Cauchy as it consists of infinite elements from each ball. We get contradiction. Thus f is almost bounded.
\end{proof}
\begin{remark}
	It is clear that every almost bounded function need not be bounded. Note that every bounded function need not be almost bounded. For example, let $ X = \left\{ \frac{1}{n} : n \in  \mathbb{N} \right\}$ with the usual metric and let $ Y = \mathbb{N}$. Define a metric $d$ on $Y$ s.t $d(x,y) = 1$ if $ x \neq y$ o.w. 0. Let
	$ f: X \rightarrow Y$ be defined as:
	$$ f\left(\frac{1}{n}\right) = n\quad \forall~n\in \mathbb{N}$$
\end{remark}
\noindent In the literature one can find another important generalization of Cauchy sequences (other than cofinally Cauchy sequences), known as pseudo-Cauchy sequences. These sequences play an important role in characterizing Atsuji spaces (discovered by Toader \cite{[T]}): a metric space is Atsuji if and only if each pseudo-Cauchy sequence of distinct points in the space clusters. In \cite{[AK1]}, authors defined those functions which preserve pseudo-Cauchy sequences (called PC-regular functions) and characterized Atsuji space as: a metric space is Atsuji if and only if each continous function defined on it is PC-regular. Our next result gives equivalent characterizations of totally bounded metric spaces using almost bounded and PC-regular functions. Before that we need to recall the following definitions.
\begin{definition}
	A sequence $(x_n)$ in a metric space $(X,d)$ is said to be pseudo-Cauchy if $\forall~\epsilon > 0$ and $n \in \N$, $\exists~k,j \in \N$ such that $k, j > n$, $k\neq j$ and $d(x_k,x_j) < \epsilon$.
\end{definition}
\begin{definition}
	A function $f: (X,d) \rightarrow (Y,\rho)$ between two metric spaces is said to be pseudo-Cauchy regular (or PC-regular for short) if ($f(x_{n}))$ is pseudo-Cauchy in $(Y,\rho)$ for every pseudo-Cauchy $(x_{n})$ in $(X,d)$.
\end{definition}
\begin{theorem}
	Let $\left(Y, \rho\right)$ be a metric space. The following statements are equivalent.
	\begin{enumerate}[label=(\alph*)]
		\item $\left(Y, \rho\right)$ is totally bounded.
		\item $~\exists$ a metric space $\left(X, d\right)$ with non-empty set of limit points, such that every function $f: \left(X,d\right) \rightarrow \left(Y,\rho \right)$ is PC-regular.
		\item $~\exists$ a metric space $\left(X, d\right)$ with non-empty set of limit points, such that every function $f: \left(X,d\right) \rightarrow \left(Y,\rho \right)$ is almost bounded.
	\end{enumerate}
\end{theorem}

\begin{proof}
	$(a)\Rightarrow (b):$ Since $Y$ is totally bounded, every sequence in $\left(Y, \rho\right)$ is pseudo Cauchy sequence. Hence every function $f:\left(X,d\right) \rightarrow \left(Y,\rho \right)$ is PC-regular.
	
	$(b)\Rightarrow (c):$ This follows from the fact that every PC-regular function is CC-regular and every CC-regular function is almost bounded.
	
	$(c)\Rightarrow (a):$ Let $ \exists  \left(X,d\right)$ s.t. every $f: \left(X,d\right) \rightarrow \left(Y,\rho \right)$ is almost bounded. Suppose $Y$ is not totally bounded, thus $\exists$ a sequence $\left(y_n\right)$ and $\epsilon > 0 $ s.t. $\rho\left(y_i, y_j\right) \geq \epsilon ~\forall~ i,j \in \mathbb{N}$. Since $X' \neq \phi$, therefore there exists a cauchy sequence $\left(x_n\right) in~X$ of distinct points. Define a function $f$ from $ X \rightarrow Y$ as 
	$$f(x)=\left\{\begin{array}{ll}y_n & \hbox{$:~if~x=x_n$ for some $n$}\\y_1 & \hbox{: otherwise}
	\end{array}\right.$$
	According to the hypothesis, the function is almost bounded, therefore $\exists~ \delta > 0$ such that $\forall~A \subseteq X$, if $d\left(A\right)< \delta$, then $f\left(A\right) \subset \displaystyle\bigcup_{i=1}^{n}B\left(z_{i}, \frac{\epsilon}{2}\right)$ for some $ z_1,z_2,...z_i \in Y$. Also, $\exists ~n_0 \in \mathbb{N}$ s.t. $\forall~n \geq n_0, ~d\left(x_n,x_m\right) < \delta~ \forall ~n,m \geq n_0$.\\
	Let  $A  = \left\{x_i : ~i\geq n_0 \right\}$. Since  $d\left(A\right)< \delta, ~\exists ~ z_1,z_2,...z_n \in Y$ s.t. $f\left(A\right) \subset \displaystyle\bigcup_{i=1}^{n}B\left(z_{i}, \frac{\epsilon}{2}\right)$, which is not possible because  $d\left(y_i,y_i\right) \geq\epsilon ~\forall~ i,j\in ~\mathbb{N}$. Thus $Y$ is totally bounded.
\end{proof}
\medskip

\section{Almost Unifomly Continuous Functions}
It is well known that continuous functions are characterized by convergent sequences, whereas uniformly continuous functions are characterized by asymptotic sequences. In \cite{[AK2]}, the authors generalized the notion of asymptotic sequences. Before proceeding futher, let us recall few definitions.
\begin{definitions}
	A pair of sequences $(x_n)$ and $(y_n)$ in a metric space $(X,d)$ is said to be:
	\begin{itemize}
		\item[$(a)$] \emph{asymptotic}\index{asymptotic sequences}, written $(x_n)\asymp(y_n)$, if $\forall$ $ \epsilon >0$, $\exists$ $ n_o \in \N$ such that $d(x_n,y_n)< \epsilon$ $\forall$ $ n > n_o$.
		\item[$(b)$] \emph{uniformly asymptotic}\index{uniformly asymptotic sequences}, written $(x_n)\asymp^u(y_n)$, if $\forall$ $ \epsilon >0$, $\exists$ $ n_o \in \N$ such that $d(x_m,y_n)< \epsilon$ $\forall$ $ m, ~n > n_o$.
		\item[$(c)$] \emph{cofinally asymptotic}\index{cofinally asymptotic sequences}, written $(x_n) \asymp_c (y_n)$, if $\forall$ $ \epsilon>0,$ $ \exists$ an infinite subset $N_\epsilon$ of $\mathbb{N}$ such that $d(x_n, y_n)< \epsilon$ $ \forall$ $ n \in N_\epsilon$.
		\item[$(d)$] \emph{cofinally uniformly asymptotic}\index{cofinally uniformly asymptotic \\sequences}, written $(x_n)\asymp^u_c (y_n)$, if $\forall$ $ \epsilon>0,$ $ \exists$ an infinite subset $N_\epsilon$ of $\mathbb{N}$ such that $d(x_n, y_m)< \epsilon$ $ \forall$ $ n, ~m \in N_\epsilon$.
	\end{itemize}
\end{definitions}
In \cite{[S]}, Snipes gave a nice characterization of Cauchy-continuous functions in terms of pairs of uniformly asymptotic sequences: $f: (X,d) \rightarrow (Y,\rho)$ between two metric spaces is Cauchy-continuous if and only if $(x_n) \asymp^u(z_n)$ in $(X,d)$ implies $(f(x_n))\asymp^u(f(z_n))$ in $(Y,\rho)$.
It is easy to see that if $f$ is continuous and satisfies the condition: $(x_n) ~\displaystyle{\asymp_c^{u}}~(z_n)$ implies $(f(x_n))~ \displaystyle{\asymp_c^{u}}~f((z_n))$, then f is almost uniformly continuous. But the converse is not true. 
\begin{example}
	Let $A = \left\{ \frac{1}{n} : n\in\mathbb{N}\right\}$ endowed with the usual distance metric and let $f: A\rightarrow\mathbb{R}$ be defined as:
	$$f(1/n)=\left\{\begin{array}{ll}1 & \hbox{:~$n$ is odd}\\2 & \hbox{:~$n$ is even}
	\end{array}\right.$$
	
	$f$ is almost uniformly continuous, but the conditions is not satisfied (take $x_n = \frac{1}{2n-1}$ and $z_n = \frac{1}{2n}).$
\end{example}
Note that the above function also shows that almost uniformly continuous functions on totally bounded spaces need not to be uniformly continuous unlike Cauchy continuous functions. Infact in \cite{[AK2]}, it is proved that every CC-regular( thus almost bounded) function on $(X,d)$ with values in $(Z,\rho)$ is uniformly continuous if and only if $(X,d)$ is uniformly discrete.

\noindent In \cite{[AK2]} the authors have given sequential characterization of CC-regular functions. Since we have proved that the class of almost bounded functions and the class of CC-regular functions is same, the following statements are equivalent:
\begin{enumerate}
	\item $f$ is almost bounded.
	\item If $(x_n)$ is cofinally Cauchy sequence in $(X,d)$ and $(f(x_n))\asymp(f(z_n))$, where $(z_n)$ is any sequence in $X$, then $(f(x_n)) ~\displaystyle{\asymp_c^{u}}~(f(z_n))$.
	\item If $(x_n) ~\displaystyle{\asymp_c^{u}}~(z_n)$ and $(f(x_n))\asymp(f(z_n))$, then $(f(x_n)) ~\displaystyle{\asymp_c^{u}}~(f(z_n))$.
\end{enumerate}
\begin{remark}
	Note that if we replace $(f(x_n))\asymp(f(z_n))$ by $(f(x_n))\asymp_c(f(z_n))$ in (b), still the condition will imply the function to be \textit{CC-regular} but the converse is not true. For example, define identity function on $X = \left\{ \frac{1}{n}, n\in \mathbb{N}\right\}$ endowed with the usual distance metric and let $(x_n)$ be the sequence $1,\frac{1}{2},2,\frac{1}{3},3,\frac{1}{4},4,...$ and $(z_n)$ be the sequence $1,10,2,10,3,10,4,...,$ here the function is almost uinformly continuous and $(x_n)$ is \textit{CC-regular} such that $(f(x_n))\asymp_c(f(z_n))$ but $(f(x_n)) ~\displaystyle{\not\asymp_c^{u}}~(f(z_n)).$
\end{remark}
Now we want to characterize almost uniformly continuous functions using sequences. For this purpose, we state the following sequential criteria of continuous functions using the generalization of asymptotic sequences.
\begin{proposition}
	Let $f: (X,d) \rightarrow (Y,\rho)$ be a function. The following are equivalent:
	
	a) $f$ is continouous.
	
	b) Whenever $(x_n)\asymp_c(x)$ then $(f(x_n))~{\asymp}_c~(f(x))$.
	
	c) Whenever $(x_n)\asymp(x)$ then $(f(x_n))~{\asymp}_c~(f(x))$
\end{proposition}
\begin{proof}
	(a)$\implies$(b) Since $(x_n)\asymp_c(x),$ there exists a subsequence of $(x_n)$ which converges to $x$. Using the continuity of $f$, $(f(x_n))\asymp_c(f(x)).$
	
	(b)$\implies$(c) Immediate.
	
	(c)$\implies$(a) Suppose $f$ is not continuous at $x$. Thus, $\exists~\epsilon > 0$ such that $\forall~ n\in\mathbb{N}, \exists~x_n$ such that $d(x_n,x) < \frac{1}{n}$ but $\rho(f(x_n),f(x)) \geq \epsilon$ which is a contradiction.
\end{proof}
Using the above result and the sequential criteria of CC-regular functions, we can prove the following result.
\begin{proposition}
	Let $f: (X,d) \rightarrow (Y,\rho)$ be a function between two metric spaces. The following are equivalent:
	\begin{enumerate}[label=(\alph*)]
		\item  $f$ is almost uniformly continuous.
		\item If $(x_n)$ is cofinally and $(f(x_n))\asymp(f(z_n))$ then $(f(x_n)) ~\displaystyle{\asymp_c^{u}}~(f(z_n))$. Also, if $(y_n)~{\asymp}_c~(y)$ then $(f(y_n))\asymp_c(f(y)).$
		\item If $(x_n) ~\displaystyle{\asymp_c^{u}}~(z_n)$ and $(f(x_n))\asymp(f(z_n))$, then $(f(x_n)) ~\displaystyle{\asymp_c^{u}}~(f(z_n))$. Also, if $(y_n)~\asymp~(y)$ then $(f(y_n))\asymp_c(f(y)).$
	\end{enumerate}
\end{proposition}

\noindent	We have found that every continuous function on a metric space is almost bounded if and only if the metric space is cofinally complete. Our next result characterizes those metric spaces $(X, d)$ such that every almost bounded function from $(X, d)$ to any metric space $(Y, \rho)$ is continuous.
\begin{theorem}
	Let $(X, d)$ be a metric space. Every almost bounded function from $(X,d)$ to any metric space $(Y, \rho$) is continuous if and only if $(X,d)$ is a discrete space.
\end{theorem}
\begin{proof}
	If $(X,d)$ is a discrete space then every function on it will be continuous. Conversely, Suppose $X$ is not discrete then $\exists~x_n\in X$ such that all are distinct and $x\in X$ such that $x_n$ converges to $x $ (take $x_n\neq x~\forall~n$). Define $f$ such that
	
	$$f(y)=\left\{\begin{array}{ll}1 & \hbox{$:~if~y=x_n$ for some $n$}\\0 & \hbox{: otherwise}
	\end{array}\right.$$
	
	$f$ is almost bounded but not continuous.
\end{proof}

\noindent Our next result states that a real valued almost uniformly continuous function having bounded range can be extended from a closed subset of a metric space to the whole metric space.
\begin{theorem} (Extention Theorem): Let $(X,d)$ be a metric space. If $f: C\rightarrow [a,b]$ where $C$ is a closed subset of $X$ and $a,b\in \mathbb{R}$ is almost uniformly continuous then f can be extended to $F:X\rightarrow[a,b]$, where $F$ is almost uniformly continuous and $F|_C = f$. The statement is true even if we replace $[a,b]$ by $(a,b)$.
\end{theorem}
\begin{proof}
	Every continuous function from $X$ whose range is totally bounded is almost uniformly continuous because every sequence in totally bounded set is cofinally cauchy. Thus by Tietze's extension theorem the result holds.
	
\end{proof}
\noindent Using the above result, we get a new characterization of cofinally complete metric spaces.
\begin{theorem}
	Let $(X,d)$ be a metric space. $X$ is cofinally complete if and only if whenever $f:X\rightarrow\mathbb{R}$ is an almost uniformly continuous function such that $f$ is never zero, then $\frac{1}{f}$ is almost bounded, in particular $\frac{1}{f}$ is almost uniformly continuous.
\end{theorem}
\begin{proof}
	If $X$ is cofinally complete, then the result is true since continuity of $f$ implies the continuity of $\frac{1}{f}$ and every continuous function on $X$ is almost uniformly continuous. Conversely, suppose $(X,d)$ is not cofinally complete. Therefore, $\exists~(x_n)\in X$ which is cofinally Cauchy but does not cluster. Thus, $A = \left\{x_n ~:~ n\in\mathbb{N}\right\}$ is discrete and closed. Define
	$$f: A \longrightarrow (0, 2)$$
	$$x_n\longmapsto \frac{1}{n}$$
	$f$ is an almost uniformly continuous function.
	
	Using the previous theorem, extend $f$ to the whole space $X$ i.e. $F: X\rightarrow (0,2)$. Cleary $F$ is never zero but $\frac{1}{F}$ is not CC-regular as $(x_n)$ is cofinally Cauchy but $(n)$ is not.
\end{proof}
\noindent We have already noted that every uniformly continuous function between two metric spaces is almost uniformly continuous. We now look for the conditions under which the converse is also true. We observe that on some well-known class of metric spaces, the class of almost uniformly continuous functions is contained in the class of uniformly continuous functions.
\begin{theorem}
	Let $\left(X, d\right)$ be a metric space. The following statements are equivalent.
	\begin{enumerate}[label=(\alph*)]
		\item $\left(X, d\right)$ is an Atsuji space.
		\item Every almost uniformly continuous function from $\left(X, d\right)$ to a metric space $\left(Y, \rho\right)$ is uniformly continuous.
		\item Every real valued almost uniformly continuous function on $X$ is uniformly continuous.
	\end{enumerate}
\end{theorem}
\begin{proof}
	$(a)\Rightarrow (b):$ Since every continuous function from an Atsuji space to any metric space is uniformly continuous, the statement (b) is true.
	
	$(b)\Rightarrow (c):$ This is immediate.
	
	$(c)\Rightarrow (a):$ First we prove that $X$
	is complete. Suppose $(x_n)$ is a Cauchy sequence of distinct points in $X$ such that it does not converge. Thus the set $A=\{x_n : n \in \N \}$ is closed and discrete. Define a function $f$ from $A$ to $\mathbb{R}$ as follows.
	$$f(x_n)=\left\{ \begin{array}{lll}
	1     & : & \mbox{$n$ is odd}\\
	2     & :&  \mbox{$n$ is even}
	\end{array}\right.$$
	Clearly, $f$ is almost uniformly continuous function on $A$. By extension theorem, extend it to an almost uniformly continuous function $F$ from $X$ to $\mathbb{R}$. Note that the function $F$ is not uniformly continuous, we arrive at a contradiction.\\
	Let $(x_n)$ be a sequence of distinct points in $X$ such that $\lim\limits_{n \to \infty} I(x_n) = 0$. Suppose the sequence does not cluster. Since $X$ is complete, $(x_n)$ has no Cauchy sequence. Thus $\exists \delta > 0$ such that by passing to a subsequence $B(x_n, \delta)$ is disjoint $\forall n \in N$. Also, $\lim\limits_{n \to \infty} I(x_n) = 0$ implies that we can assume $I(x_n) <$ min$\{\delta, \frac{1}{n}\}~\forall n \in \N$. Let $\delta_n$ = min$\{\delta, \frac{1}{n}\}$ $\forall n \in \N$ . Consequently, we can choose $y_n \neq x_n$ such that $d(x_n, y_n) < \delta_n$. Therefore the sequence $(y_n)$ does not have any cluster point and the set $A=\{x_n, y_n : n \in \N\}$ is closed and discrete in $X$. Define a function $f$ as follows
	$$f(x)=\left\{ \begin{array}{lll}
	1     & : & \mbox{if $x= x_n$ for some $n$}\\
	2     & :&  \mbox{if $x= y_n$ for some $n$}
	\end{array}\right.$$
	The almost uniform extension of $f$ to the whole space $X$ is almost uniformly continuous but not uniformly continuous, which is a contradiction. Thus $(X, d)$ is an Atsuji space.
\end{proof}

\noindent Our next results studies those metric spaces on which each almost uniformly continuous function is bounded.
\begin{theorem}
	Let $(X,d)$ be a metric space. Then the following are equivalent:
	\begin{itemize}
		\item[$(a)$] The metric space $(X,d)$ is totally bounded.
		\item[$(b)$] Whenever $(Y,\rho)$ is a metric space and $f: (X,d) \rightarrow (Y,\rho)$ is almost uniformly continuous, then f is bounded.
		\item[$(c)$] Whenever $f: (X,d) \rightarrow (Z,\sigma)$ is  almost uniformly continuous where $(Z, \sigma)$ is an unbounded metric space, then f is bounded.
		\item[$(d)$] Whenever $f$ is a real valued almost uniformly continuous function, f is bounded.
	\end{itemize}
\end{theorem}
\begin{proof}
	$(a)\Rightarrow (b):$ Since $f$ is almost bounded, $\forall~\epsilon > $ 0 $\exists$ a $\delta >$ 0 such that for every $A \subseteq X$, if the diameter of $A$ is $< \delta$ then there is a finite subset $B = \left\{b_{1}, b_{2}, ...,b_{n}\right\}$ of $Y$ such that $f(A) \subseteq \displaystyle\bigcup_{i=1}^{n}B(b_{i}, \epsilon)$. As $X$ is totally bounded, $X \subseteq \displaystyle\bigcup_{i=1}^{m}B(x_{i}, \delta/2)$ for some $x_1, x_2,..., x_n \in X$. Thus $f(X)$ is contained in finite union of open balls of radius $\epsilon$ in $Y$. 
	
	$(b)\Rightarrow (c):$ This is immediate.
	
	$(c)\Rightarrow (d):$ This is immediate.
	
	$(d)\Rightarrow (a):$ Suppose $X$ is not totally bounded. Therefore, $\exists \delta > 0$ and a sequence $(x_n)$ such that $d(x_n, x_m) > \delta~ \forall n,~m \in \N$. Consequently, the function $f: (X,d) \rightarrow \mathbb{R}$ defined as:
	$$f(x)=\left\{ \begin{array}{lll}
	n- \frac{2n}{\delta}d(x, x_n)     & : & x\in B(x_n, \frac{\delta}{2}) \mbox{~for some}~ n \in \N\\
	~~~~~~~	0     & :& otherwise
	\end{array}\right.$$
	is almost uniformly continuous but not bounded, a contradiction.
\end{proof}
\noindent A metric space $(X, d)$ is UC if and only if each sequence $(x_n)$ in $X$ satisfying $\lim\limits_{n \to \infty} I(x_n) = 0$ clusters, where $I(x)= d(x,X\setminus \{x\})$ measures the isolation of $x$ in the space \cite{[MA],[KJ1]} . Similarly, a metric space $(X, d)$ is cofinally complete if and only if each sequence $(x_n)$ in $X$ satisfying $\lim\limits_{n \to \infty} \nu(x_n) = 0$ clusters, where $\nu(x)= sup\{\epsilon>0: cl(B_d(x,\epsilon))$ is compact$\}$ if $x$ has a compact neighborhood, and $\nu(x)=0$ otherwise \cite{[B1]}. Now we are ready to give few more results.
\begin{proposition}
	Let $f$ be a function from a metric space $(X, d)$ to $(Y,\rho)$ such that it is uniformly continuous on $\left\{x \in X : I(x) < \epsilon \right\}$ for some $\epsilon > 0$, then $f$ is almost uniformly continuous.
\end{proposition}
\begin{proof}
	Let $(x_n)$ be a sequence in $X$ converging to a point $x \in X$. Thus, there exists $n_0 \in \mathbb{N}$ such that $d(x_n, x_m) < \frac{\epsilon}{2}~~\forall~n,~m\geq n_0$. Thus, $f$ is uniformly continuous on the set $\left\{x, x_n~:~n\geq n_0\right\}$ which implies that $(f(x_n))$ converges to $f(x)$. Hence the function is continuous. \\
	Now suppose the function is not almost bounded. Therefore $\exists \epsilon_o > 0$ such that $\forall n \in \mathbb{N}$, $\exists x_n \in X$ such that $f(B(x_n, \frac{1}{n}))$ can not be written as finite union of open  balls of radius $\epsilon_o$ in $Y$. Using the same technique as in \ref{T1}, we can choose a sequence consisting infinite elements from each ball $B(x_n, \frac{1}{n})$ such that the image of the sequence is $\frac{\epsilon_o}{2}$ discrete. Choose $\frac{1}{n_o} < \epsilon$, since the function is uniformly continuous on the set $A = \left\{x \in X : I(x) < \epsilon \right\}$, $\exists \delta > 0$ such that $\forall x, y \in A$ and $d(x,y) < \delta$, $\rho(f(x), f(y)) < \frac{\epsilon_o}{2}$. Choosing $\frac{1}{m} < \mbox{min} \{\frac{1}{n_o}, \delta\}$, we get contradiction.
\end{proof}
\begin{remark}
	The converse of the above result is not true. Also if we replace the isolation functional with local compactness functional, the result won't be true, e.g., take any discontinuous function from $\mathbb{R}$ to $\mathbb{R}$, but if the function $f$ is continuous then the result will be true and can be proved in a similar manner.
\end{remark}
\begin{proposition}
	Let $f$ be a continuous function from a metric space $(X, d)$ to $(Y,\rho)$ such that it is uniformly continuous on $\left\{x \in X : \nu(x) < \epsilon \right\}$ for some $\epsilon > 0$, then $f$ is almost uniformly continuous.
\end{proposition}
\begin{proposition}
	Let $(X, d)$ be a metric space such that every continuous function is uniformly continuous on $\left\{x \in X : \nu(x) < \lambda \right\}$ for some $\lambda > 0$, then $(X, d)$ is cofinally complete.
\end{proposition}
\begin{proof} Let $(x_n)$ be a cofinally Cauchy sequence of distinct points in $X$. By Proposition 2.4 in \cite{[B1]}, there exists a pairwise disjoint family $\{\mathbb{M}_j: j \in \N\}$ of infinite subsets of $\N$ such that if $i \in \mathbb{M}_j$ and $l \in \mathbb{M}_j$ then $d(x_i, x_l)< \frac{1}{j}$. Suppose the sequence $(x_n)$ does not cluster, which implies the set $A =\{x_n : n \in \N \}$ is closed and discrete. Enumerate the elements of $\mathbb{M}_n$ as $l^n_1, l^n_2, l^n_3,...$ for each $n \in \N$. Therefore, for each $n$, there exists unique $i,~k \in \N$ such that $n= l^k_i$. Define a function $f$ on $A$ such that $f(x_{l^k_i}) = 1$ if $i$ is even and $f(x_{l^k_i}) = 2$ if $i$ is odd. Since $f$ is almost uniformly continuous on $A$, extend it to an almost uniformly continuous function $F$ on the whole space $X$.
	
	We claim $\nexists$ any $\lambda > 0$ such that $F$ is uniformly continuous on $\left\{x \in X : \nu(x) < \lambda \right\}$. Suppose the claim is not true. Therefore $\exists \lambda > 0$ such that $F$ is uniformly continuous on $\left\{x \in X : \nu(x) < \lambda \right\}$. Choose $\frac{1}{n_o} < \lambda$, consequently $\nu(x_k) < \frac{1}{n_o} < \lambda~ \forall x_k \in \displaystyle\bigcup_{n \geqslant n_o}\mathbb{M}_n$ but $F$ is not uniformly continuous on $\displaystyle\bigcup_{n \geqslant n_o}\mathbb{M}_n$, we arrive at a contradiction.
\end{proof}
\medskip
\section{Maps that preserve cofinally complete spaces}
In general, a continuous image of a cofinally complete space need not be a cofinally complete space. The following example shows that a homeomorphic image of a cofinally complete space need not be cofinally complete space.
\begin{example} Let $X = \N$ and $Y=\{1/n : n \in \N \}$ both endowed with the usual distance metric. Let $f: X \rightarrow Y$ be defined as $f(n) = \frac{1}{n}~ \forall n \in \N$. Clearly, $f$ is homeomorphism but $Y$ is not a cofinally complete space.
\end{example}

Recall that a bijection $f$ from a metric space $(X,d)$ to another metric space $(Y, \rhd)$ is called homeomorphism if both $f$ and $f^{-1}$ are continuous and it is called uniform homeomorphism if both $f$ and $f^{-1}$ are uniformly continuous. In relation to these definitions, we define the following.
\begin{definition}
	A bijection $f$ from a metric space $(X,d)$ to another metric space $(Y, \rhd)$ is called almost uniform homeomorphism if both $f$ and $f^{-1}$ are almost uniformly continuous. If there exists such a bijection between two metric spaces, then the spaces are called almost uniformly homeomorphic.
\end{definition}

\begin{proposition}
	If two metric spaces $(X,d)$ and $(Y, \rho)$ are almost uniformly homeomorphic, then $(X,d)$ is cofinally complete if and only if $(Y,\rho)$ is cofinally complete.
\end{proposition}
\begin{definition}
	A subset $A$ of a metric space $(X, d)$ is called cofinally complete subspace of $X$ if every cofinally Cauchy sequence in $A$ clusters in $A$.
\end{definition}
\begin{proposition}
	Let $(X,d)$ be a metric space and $A$ be its subset. The following statements are equivalent.
	\begin{enumerate}[label=(\alph*)]
		\item $A$ is a cofinally complete subspace of $X$.
		\item Whenever $(x_n)$ is a sequence in $X$ with $\lim\limits_{n \to \infty} \nu_A(x_n) = 0$ then $(x_n)$ has a cluster point in $A$.
	\end{enumerate}
\end{proposition}
\begin{remark}
	Note that if $A$ is a cofinally complete subspace of a metric space $(X,d)$ and $\lim\limits_{n \to \infty} \nu(x_n) = 0$ for some sequence $(x_n)$ in $A$, then the sequence need not cluster. For example, consider the metric space $X=\{(n-1/2, n+1/2) \cap \mathbb{Q} : n \in \N\}$ endowed with the usual distance metric. The set $A= \{n : n \in \N\}$ is a cofinally complete subspace of $X$ such that $\lim\limits_{n \to \infty} \nu(n) = 0$ but the sequence $(n)$ does not have any cluster point in $A$. Also, if every sequence $(x_n)$ in a subset $A$ of a metric space $(X,d)$ with $\lim\limits_{n \to \infty} \nu(x_n) = 0$ clusters in $A$, then $A$ need not be a cofinally complete metric space. For example, let $X=\{0,~ 1/n : n \in \N\}$ endowed with usual distance metric and let $A=\{1/n : n \in \N\}$. There is no sequence $(x_n)$ in $A$ such that $\lim\limits_{n \to \infty} \nu(x_n) = 0$, but $A$ is not a cofinally complete subspace.
\end{remark}
\begin{definition}
	A function $f$ from a metric space $(X,d)$ to a metric space $(Y,\rho)$ is said to be CC-preserving if for every cofinally complete subspace $A$ of $X$, $\overline{f(A)}$ is a cofinally complete subspace of $Y$.
\end{definition}
\begin{theorem}
	Let $f$ be a continuous function from a metric space $(X,d)$ to a metric space $(Y, \rho)$. Then the following statements are equivalent.
	\begin{enumerate}[label=(\alph*)]
		\item $f$ is CC-preserving.
		\item If a subset $A$ of $X$ is a uniformly locally compact in its relative topology such that $\overline{f(A)}$ is discrete in $Y$, then $\overline{f(A)}$ is uniformly locally compact in its relative topology. 
	\end{enumerate}
\end{theorem}
\begin{proof}
	$(a)\Rightarrow (b):$ Let $z \in \overline{f(A)}$. Suppose $\nu_{\overline{f(A)}}(z) = 0$, thus $I_{\overline{f(A)}}(y_n) =0$, but $\overline{f(A)}$ is discrete in $Y$, a contradiction. Thus, $\nu_{\overline{f(A)}}(z) > 0~\forall z \in \overline{f(A)}$.
	Now suppose $\overline{f(A)}$ is not uniformly locally compact in its relative topology. Therefore, $\exists$ a sequence $(y_n)$ in $\overline{f(A)}$ such that $\nu_{\overline{f(A)}}(y_n) < 1/n~\forall n \in \N$. Since $A$ is cofinally complete subset of $X$, $\overline{f(A)}$ is cofinally complete subset of $Y$. Thus, the sequence $(y_n)$ has a cluster point in $\overline{f(A)}$, a contradiction. Hence $\overline{f(A)}$ is not uniformly locally compact in its relative topology.
	
	$(b)\Rightarrow (a):$ Suppose $A$ is a cofinally complete subset of $X$ such that $\overline{f(A)}$ is not cofinally complete. Thus there exists a cofinally Cauchy sequence $(z_n)$ of distinct points in $\overline{f(A)}$ which does not cluster. Since $(z_n)$ has no cluster point, we can find $\epsilon_1 > 0$ such that $\rho(z_1, z_m) > \epsilon_1 ~\forall m > 1$. We can choose $x_1 \in A$ such that $\rho(f(x_1), z_1) < min\{\epsilon_1, 1 \}$. Suppose we have chosen $x_1, x_2,...x_n \in A$ such that for each $i$, $1\leq i \leq n$, $\rho(f(x_i), z_i) < \mbox{min}\{\epsilon_i, 1/i \}$ where $\rho(z_i, z_m) > \epsilon_i ~\forall m \neq i$. Choose $\epsilon_{n+1} > 0$ such that $\rho(z_{n+1}, z_m) > \epsilon_{n+1} ~\forall m \neq n+1$. Since $z_{n+1} \in \overline{f(A)}$, we can choose $x_{n+1} \in A$ such that $\rho(f(x_{n+1}), z_{n+1}) < min\{\epsilon_{n+1}, 1/{n+1} \}$. So by induction we can choose a sequence $(x_n)$ of distinct points in $A$ such that $f((x_n))$ has distinct points. \\
	Since $(z_n)$ does not cluster, $(x_n)$ also has no cluster point in $A$. Thus the set $E=\{x_n : n \in \N \}$ is a closed and discrete subset of $A$ in $X$. Thus $\nu_E(x_n) > 0~\forall n \in \N$. Since $E$ is a closed subset of $A$, $\nu_A(x) \leq \nu_E(x)$. Suppose there exists a sequence of distinct points $(a_n)$ in $E$ such that $\nu_E(a_n) < 1/n ~\forall n \in \N$, then $\nu_A(a_n) < 1/n ~\forall n \in \N$. Since $A$ is cofinally complete subset, $(a_n)$ has a cluster point, which is a contradiction. Hence, $E$ is uniformly locally compact in its relative topology such that $\overline{f(E)}$ is closed and discrete subset of $Y$. Thus, $\overline{f(E)}$ is uniformly locally compact in its relative topology, but $(f(x_n))$ is a cofinally Cauchy sequence of distinct points which does not cluster, consequently $\overline{f(E)}$ can not be uniformly locally compact in its relative topology, a contradiction.
\end{proof}
\begin{theorem}
	Let $(X,d)$ be a cofinally complete metric space and $(Y, \rho)$ be a metric space. Let $f$ be a continuous map from $(X, d)$ onto $(Y, \rho)$. Then the following statements are equivalent.
	\begin{enumerate}[label=(\alph*)]
		\item $(Y, \rho)$ is a cofinally complete space.
		\item For any sequence $(x_n)$ in $X$ with inf$~\nu(x_n) > 0$, either $(f(x_n))$ has a cluster point or the set $M=\{n \in \N : f(x_n) \in nlc(Y)\}$ is finite and inf $\{\nu(f(x_n)) : n \in \mathbb{N}\setminus M \} > 0$.
		\item $f$ is CC-preserving.
	\end{enumerate}
\end{theorem}
\begin{proof}
	$(a)\Rightarrow (b):$ This follows from the fact that a metric space $(X,d)$ is cofinally complete if and only if whenever $(x_n)$ is a sequence in $X$ with $\lim\limits_{n \to \infty} v(x_n) = 0$, then $(x_n)$ has a cluster point. 
	
	$(b)\Rightarrow (c):$ Suppose $A$ is a cofinally complete subset of $X$ such that $\overline{f(A)}$ is not cofinally complete. Choose the sequences $(z_n), (f(x_n))$ and $(x_n)$ as chosen in Theorem 5. Since $X$ is cofinally complete and $(x_n)$ has no cluster point, $M = \{n \in \N : \nu(x_n) = 0 \}$ is finite and  $\inf\limits_{n \notin M} \nu(x_n) > 0$. According to the hypothesis, inf $\{v(f(x_n)) : n \in \mathbb{N}\setminus M \} > 0$, which is not possible since $(f(x_n))$ is a cofinally Cauchy sequence which does not cluster. Hence $f$ is CC-preserving.
	
	$(c)\Rightarrow (a):$ This is immediate.
\end{proof}

\medskip

\section{Cofinal Completion}
We have defined the following terms to find equivalent characterisations of a metric space to have cofinal completion.
	\begin{definition}
		If x is an element of a metric space $(X,d)$ such that it has no totally bounded neighborhood, set $t(x)=0,$ otherwise, put $t(x)= sup\{\epsilon>0: B_d(x,\epsilon)\}$ is totally bounded$\}$. This geometric functional is called the local total boundedness functional on $X$. 
		\indent The set $\{x \in X : t(x) =0\}$ is the set of points of non-local total boundedness of $X$, which is denoted by $nlt(X)$. Thus a metric space is said to be locally totally bounded if $t(x) > 0~\forall x \in X$, while it is called uniformly locally totally bounded if inf$\{t(x) : x \in X\} > 0$.
	\end{definition}
	\begin{definition}
		Let $(X, d)$ be a metric space and $A$ be any subset of $X$. Then $\overline{t}(A)= sup\{t(x) : x \in A\}$ and $\underline{t}(A)= sup\{t(x) : x \in A\}$.
	\end{definition}
	\begin{definition}
		Let $(X,d)$ be a metric space. A real valued function $f$ on $X$ is said to be t-bounded if $\exists~r > 0$ such that $\{ f(x) : x \in X,~t(x) < r \}$ is bounded.
	\end{definition}
	\begin{definition}
		Let $(X,d)$ be a metric space. A real valued function $f$ on $X$ is said to be t-bounded on a proper subset $A$ of $X$ if~$\exists ~r>0$ such that the set $\{ f(x) : x \in A,~t(x) < r \}$ is bounded.
	\end{definition}
	\begin{definition}
		Let $(X,d)$ be a metric space and $A$ be a non-empty subset of $X$. Then $A$ is called nowhere locally totally bounded if $\forall \epsilon > 0$, the set $\{a\in A: t(a) \geq \epsilon\}$ is totally bounded.
	\end{definition}
	\begin{definition}
		Let $B$ and $C$ be disjoint subsets of $(X,d)$. Then $B$ and $C$ are said to be asymptotic if $\forall \epsilon > 0, \exists b\in B, c\in C$ with $d(b,c)< \epsilon$.
	\end{definition}
	\begin{lemma}
		Let $(X,d)$ be a metric space. If $X$ is not boundedly totally bounded, then $t: X \rightarrow [0,\infty)$ is uniformly continuous.
	\end{lemma}
	\begin{proof}
		Let $\epsilon >0$ and $x, ~y \in X$ such that $d(x,y) \leq \frac{\epsilon}{2}$. We claim that $t(x) \leq t(y) + \epsilon$. Suppose $t(x)> t(y)+\epsilon$. Then there exists $r > t(y)+\epsilon$ such that $B(x, r)$ is totally bounded. Since $B(y, r- \frac{\epsilon}{2}) \subset B(x, r)$, $t(y) \geq r- \frac{\epsilon}{2}$. This implies $r > r+ \frac{\epsilon}{2}$, which is a contradiction. Thus, $t(x) \leq t(y) + \epsilon$. By interchanging the roles of $x$ and $y$, we get $|t(x)-t(y)| \leq \epsilon$.
	\end{proof}

	\begin{theorem}
		Let $(X,d)$ be a metric space. The following statements are equivalent:
		\begin{enumerate}[label=(\alph*)]
			\item The completion $(\widehat{X},d)$ of $(X,d)$ is cofinally complete.
			\item Whenever $(x_n)$ is a sequence in $X$ with $\lim\limits_{n \to \infty} t(x_n) = 0$, then $(x_n)$ has a Cauchy subsequence.
			\item The set $nlt(X)$ is totally bounded and for any sequence $(x_n)$ in $X \setminus nlt(x)$ with no Cauchy subsequence, inf $t(x_n) > 0$.
			\item Let $A$ be a subset of $X$ such that no infinite subset of $X$ is totally bounded, then $ A \cap nlt(X)$ is finite and  $inf \{ t(x): x\in A\setminus nlt(X)\} > 0$.
			\item For any completely discrete subset $A$ of $X$, $A \cap nlt(X)$ is finite and $inf \{ t(x): x\in A\setminus nlt(X)\} > 0$.
			\item Let $(x_n)$ be a sequence in $(X,d)$ with no Cauchy subsequence, then $\exists~ n_0 \in \mathbb{N}$ such that $inf \{ t(x_n): n \geq n_0 \} > 0$.
			\item The set $nlt(X)$ is totally bounded and if $U$ is any infinite complete subset of $X$ disjoint from $nlt(X)$, then $U$ is uniformly locally totally bounded.
		\end{enumerate}
	\end{theorem}
	\begin{proof}
		$(a)\Rightarrow (b):$ Suppose the set $\{x_n : n \in \N\}$ is not totally bounded. Therefore, $\exists \epsilon > 0$ and a subsequence $(x_{n_k})$ of $(x_n)$ such that the open balls $B(x_{n_k}, \epsilon)$ are pairwise disjoint. Rename the subsequence by $(x_n)$. Clearly, $\lim\limits_{n \to \infty} t(x_n) = 0$. By passing to a subsequence, we can assume $t(x_n) <$ min$\{\epsilon, \frac{1}{n}\}~\forall n \in \N$. Let $r_n$ = min$\{\epsilon, \frac{1}{n}\}$ $\forall n \in \N$. Thus there exists a sequence $(w_j^n)_{j \in \N}$ in $B(x_n, r_n)$ with no Cauchy subsequence. Now consider a partition $\{M_n: n \in \N\}$ of $\N$ where each $M_n$ is an infinite subset of $\N$. Define a sequence $(y_j)$ where $y_j = w_j^n$ for $j \in M_n$. Then $(y_j)$ is a cofinally Cauchy sequence in $(\widehat{X},d)$ which does not cluster, a contradiction.
		
		$(b)\Rightarrow (c):$ Clearly, the set $nlt(X)$ is totally bounded. Let $(x_n)$ be in $X \setminus nlt(x)$ with no Cauchy subsequence such that $inf\{x_n : n \in \N \} = 0$, i.e., there exists a subsequence $(x_{n_k})$ of $(x_n)$ such that $\lim\limits_{k \to \infty} t((x_{n_k})) = 0$ but it has no cluster point, a contradiction.
		
		$(c)\Rightarrow (d):$ This is immediate.
		
		$(d)\Rightarrow (e):$ Let $B$ be any infinite subset of $A$. Let $(x_n)$ be a sequence of distict points in $B$. Suppose $(x_n)$ has a Cauchy subsequence $(x_{n_k})$. Since $X$ is complete, there exists $x \in A$ such that $(x_{n_k})$ converges to $x$, which is a contradiction as $A$ is a discrete set. Thus, no infinte subset of $A$ is totally bounded.
		
		$(e)\Rightarrow (f):$ This is immediate.
		
		$(f)\Rightarrow (g):$ Suppose $U$ is not uniformly locally totally bounded. Therefore, there exists a sequence $(x_n) \in U$ such that $t(x_n) < \frac{1}{n} ~\forall n \in \N$. According to the hypothesis, $(x_n)$ has a Cauchy susbsequence say $(x_{n_k})$. Since $U$ is complete, $(x_{n_k})$ converges to some point $x \in U$, by Lemma 1,  $t(x)=0$. Thus we got contradiction as $U$ is disjoint from $nlt(X)$.
		
		$(g)\Rightarrow (a):$ Suppose $(\hat{x}_n)$ be a cofinally Cauchy sequence in $(\widehat{X},d)$ with no Cauchy subsequence. Then there exists a sequence $(x_n) \in X$ such that $d(\hat{x}_n, x_n) < \frac{1}{n}~\forall n \in \N$. Consequently, $(x_n)$ is cofinally Cauchy in $(X,d)$ with no Cauchy subsequence and hence the set $A=\{x_n : n \in \N\}$ is complete. Since $(x_n)$ has no Cauchy subsequence, $nlt(X) \cap A$ is a finite set, say $B$. Let $U= A\setminus B$. Let $\{y_n : n \in \N \}$  be an enumeration of the elements of $U$. According to the hypothesis, $\alpha = inf\{ t(y_n) : n \in \N\} > 0$. Clearly, $(y_n)$ is a cofinally Cauchy sequence. For $\alpha/2 > 0$, there exists an infinite subset $A$ of $\N$ such that $d(y_n, y_m) < \alpha/2~~\forall~n, m \in A$. Since $t(y_n) > \alpha ~~\forall n \in \N$, we arrive at a contradiction.
	\end{proof}
	\begin{theorem}
		Let $(X,d)$ be a metric space. The following statements are equivalent:
		\begin{enumerate}[label=(\alph*)]
			\item The completion $(\widehat{X},d)$ of $(X,d)$ is cofinally complete.
			\item Let $f: X \rightarrow \mathbb{R}$ be a Cauchy continuous function. Then there exists $n_0 \in \mathbb{N}$ such that every point of the set $A = \{x : |f(x)| \geq n_0 \}$ has totally bounded neighbourhood and $inf \{ t(x): x\in A\} > 0$ 
			\item Every real valued Cauchy continuous function on $(X,d)$ is t-bounded.
			\item Every real-valued Cauchy continous function on $(X,d)$ is t-bounded on every proper subset of $X$.
		\end{enumerate}
	\end{theorem}
	\begin{proof}
		$(a)\Rightarrow (b):$ Suppose $\forall n \in \N, ~\exists$ $x_n \in X$ such that $|f(x_n)| \geq n$ and $t(x_n)=0$. Since $(\widehat{X},d)$ is cofinally complete, $(x_n)$ must have a Cauchy subsequence which is a contradiction because $f$ is Cauchy continuous and every susbsequence of $(f(x_n))$ is unbounded.
		Thus, there exists $n_o \in \N$ such that the set $A=\{x:  |f(x)| \geq n_o\}$ is locally totally bounded.
		\\Since $f$ is a Cauchy continuous function from $X$ to $\mathbb{R}$, we can extend it to a function $\hat{f}$ from $\widehat{X}$ to $\mathbb{R}$ such that $\hat{f}$ is Cauchy continuous, thus continuous. Since $(\widehat{X},d)$ is cofinally complete, $\exists n_o' \displaystyle{n_{o}^{'}}$ such that for $A'=\{x \in \widehat{X} : |\hat{f}(x)| \geq \displaystyle{n_{o}^{'}} \}$, $A' \bigcap nlc(\widehat{X}) = \phi$ and $inf\{\hat{\nu}(x) : x \in A'\} > 0$. Let $m$= max$\{n_o, \displaystyle{n_{o}^{'}}\}$, $A=\{x \in X : |f(x)| \geq m\}$ and $A'=\{x \in \widehat{X} : |\hat{f}(x)| \geq m\}$. Note that $A \subseteq A'$. We claim that $\inf \{t(x): x \in A\} >0$. Suppose the claim is not true, therefore $\exists (x_n) \in A$ such that $\lim\limits_{n \to \infty} t(x_n) = 0$.\\
		Let $x \in X$ and let $t(x) < \alpha$. Therefore, there exists a sequence $(y_n) \in B(x, \alpha)$ which does not have any Cauchy subsequence . Therefore, $(y_n)$ does not have any cluster point in $(\widehat{X},d)$, which implies $\hat{\nu}(x) \leq \alpha$.\\
		Since $\lim\limits_{n \to \infty} t(x_n) = 0$, $\lim\limits_{n \to \infty} \hat{\nu}(x_n) = 0$. It is contradiction to the fact that $inf\{\hat{\nu}(x) : x \in A'\} > 0$. Thus, $\inf \{t(x): x \in A\} >0$.
		
		$(b)\Rightarrow (c):$ Let $f: (X,d) \longrightarrow (\mathbb{R}, |.|)$ be a real-valued function. Therefore, $\exists n_o \in \N$ such that for $A=\{x \in X : |f(x)| \geq n_o\}$, $inf\{t(x) : x \in A\} > 0$. Let $inf\{t(x) : x \in A\} = r$. Now, let $B$ be any subset of $X$. Then $\{f(x) : x \in B, t(x) < r\}$ is bounded because $t(x) < r$ implies $x \notin A$, thus $|f(x)| < n_o$.
		
		$(c)\Rightarrow (d):$ This is immediate.
		
		$(d)\Rightarrow (a):$ Suppose $(\hat{x}_n)$ be a cofinally Cauchy sequence in $(\widehat{X},d)$ with no Cauchy subsequence. Then there exists a sequence $(x_n) \in X$ such that $d(\hat{x}_n, x_n) < \frac{1}{n}~\forall n \in \N$. Consequently, $(x_n)$ is cofinally Cauchy in $(X,d)$ with no Cauchy subsequence and hence the set $A=\{x_n : n \in \N\}$ is complete. Since $(x_n)$ has no Cauchy subsequence, we can assume the sequence $(x_n)$ consists of distinct elements. Consider the function $f: A \longrightarrow \mathbb{R}$ where $f(x_n) = n.$ Clearly $f$ is a Cauchy continuous function. Let $g$ be its extension to the whole space $X$ such that $g$ is Cauchy continuous. Let $B = A\setminus \{x\}$. Therefore, $\exists r > 0$ such that the set $\{g(x) : x \in B, t(x) < r\}$ is bounded. Therefore, $|g(x)| \leq M~\forall x \in B$ and $t(x) < r$. Since $(x_n)$ is cofinally Cauchy sequence, $\exists$ infinite subset $\N_r$ of $\N$ such that $d(x_n, x_m) < r~~\forall n, m \in \N_r$, thus $t(x_n) < r~\forall n \in \N_r$. Hence, $|g(x_n)| \leq M~\forall n \in \N_r$. Since $(x_n)$ is a sequence of distinct points and $\N_r$ is an infinite set, we get contradiction.
	\end{proof}
	\begin{theorem}
		Let $(X,d)$ be a metric space. The following statements are equivalent:
		\begin{enumerate}[label=(\alph*)]
			\item The completion $(\widehat{X},d)$ of $(X,d)$ is cofinally complete.
			\item Whenever $F_1$ and $F_2$ are disjoint asymptotic subsets of $X$ such that $F_1$ is closed and $F_2$ is complete, there exists a $\delta > 0$ such that $F_1 \cap \{ x~\in X : t(x) > \delta  \}$ and $F_2 \cap \{ x~\in X : t(x) > \delta  \}$ are asymptotic.
			\item Whenever $F_1$ and $F_2$ are disjoint asymptotic complete subsets of $X$, there exists a $\delta > 0$ such that $F_1 \cap \{ x \in X : t(x) > \delta  \}$ and $F_2 \cap \{ x \in X : t(x) > \delta  \}$ are asymptotic.
			\item Whenever $B$ and $C$ are non empty disjoint complete subsets of $X$ such that $C$ is almost nowhere totally bounded, then $d(B,C) > 0 $.
			\item Every nonempty complete and almost nowhere locally totally bounded subset of $X$ has countable character, that is, for every nonempty complete and almost nowhere locally totally bounded subset $A$ of $X$ and for every open set $G$ containing $A$, there exists $r>0$ such that $A\subseteq B(A,r)\subseteq G$.
			\item  Whenever $B$ and $C$ are non empty disjoint closed subsets of $X$ such that $C$ is complete and almost nowhere totally bounded, then $d(B,C) > 0 $.
			\item Whenever $B$ and $C$ are non empty disjoint subsets of $X$ such that $B$ is complete and $C$ is closed and almost nowhere totally bounded, then $d(B,C) > 0 $.
			\item Whenever $F_n$ is a decreasing sequence of non-empty complete subsets of $X$ with $\lim\limits_{n \to \infty} \underline{t}(F_n) = 0$, then $\bigcap \{F_n: n\in\mathbb{N}\}$ is non empty.
			\item Whenever $F_n$ is a decreasing sequence of non-empty closed subsets of $X$ such that for some m $\in \mathbb{N}$, $F_m$ is complete and $\lim\limits_{n \to \infty} \underline{t}(F_n) = 0$, then $\bigcap \{F_n: n\in\mathbb{N}\}$ is non empty.
			\item Whenever $F_n$ is a decreasing sequence of non-empty closed subsets of $X$ such that for some m $\in \mathbb{N}$, $F_m$ is complete and $\lim\limits_{n \to \infty} \overline{t}(F_n) = 0$, then $\bigcap \{F_n: n\in\mathbb{N}\}$ is non empty.
			\item Whenever $F_n$ is a decreasing sequence of non-empty complete subsets of $X$ with $\lim\limits_{n \to \infty} \overline{t}(F_n) = 0$, then $\bigcap \{F_n: n\in\mathbb{N}\}$ is non empty.
			\item For every sequence $(A_n)$ of closed sets in $X$ such that for some $k\in \mathbb{N}$, $A_k$ is complete in $X$ and for some $t\in \mathbb{N},~ (k\neq t)$, $A_t$ is almost nowhere locally totally bounded, if $\bigcap\limits_{n=1}^{\infty}A_n=\emptyset$, then there exists $r>0$ such that $\bigcap\limits_{n=1}^{\infty}B(A_n,r)=\emptyset$.
			\item  For every sequence of complete sets $(A_n)$ in $X$ such that for some $k\in \mathbb{N}$, $A_k$ is almost nowhere locally totally bounded in $X$ and $\bigcap\limits_{n=1}^{\infty}A_n=\emptyset$, there exists $r>0$ such that $\bigcap\limits_{n=1}^{\infty}B(A_n,r)=\emptyset$.
			\item  For every sequence $(A_n)$ of closed sets in $X$ such that for some $k\in \mathbb{N}$, $A_k$ is complete and almost nowhere locally totally bounded and $\bigcap\limits_{n=1}^{\infty}A_n=\emptyset$, then there exists $r>0$ such that $\bigcap\limits_{n=1}^{\infty}B(A_n,r)=\emptyset$.
		\end{enumerate}
	\end{theorem}
	\begin{proof}
		$(a) \Rightarrow (b)$: Suppose no such $\delta$ exists. Then for each $n \in \N$, there exists $x_n \in F_1$ and $y_n \in F_2$ with $d(x_n, y_n)< \frac{1}{n}$ and $\{x_n, ~y_n\}\cap \{x: t(x) \leq \frac{1}{n}\} \neq \emptyset$. By the uniform continuity of the functional $t$, $ \lim\limits_{n \to \infty}t(x_n)= \lim\limits_{n \to \infty}t(y_n)=0$. Since $F_2$ is complete and $F_1$ is closed, condition $(b)$ of Theorem [1] implies that $(x_n)$ and $(y_n)$ has a cluster point which will lie in $F_1 \cap F_2$. We arrive at a contradiction.
		
		$(b) \Rightarrow (c)$: This is immediate.
		
		$(c) \Rightarrow (d)$: Suppose $d(B, C)=0$. According to the hypothesis, there exists $\delta >0$ such that for all $n \in \N$, $ \exists ~b_n \in B, ~c_n \in C$ with $t(b_n)> \delta$, $ t(c_n)> \delta$ and $d(b_n, c_n)< \frac{1}{n}$. Since $C$ is almost nowhere locally compact, the sequence $(c_n)$ has a Cauchy subsequence say $(c_{n_k})$ and hence $(b_{n_k})$ is a Cauchy subsequence too. Since $B$ and $C$ are complete subsets of $X$, cluster point of both the subsequences lies in $B \cap C$. We arrive at a contradiction.
		
		$(d) \Rightarrow (e)$: Suppose $A$ does not have countable character, that is, there exists an open set $G$ containing $A$ such that for each $n \in \N,~B(A, 1/n) \subsetneq G$. Let $A_1=X\setminus G$.	Then for each $n \in \N$, we can find $x_n \in A$ and $y_n \in A_1$ such that $d(x_n,y_n) < \frac{1}{n}$. If $(x_n)$ has a Cauchy subsequence $(x_{n_k})$, then since $A$ is complete, there exists $x \in A$ such that $(x_{n_k})$ converges to $x$, hence $(y_{n_k})$ also converges to $x$. Consequently $x \in A \cap A_1$, which is a contradiction. Hence neither $(x_n)$ or $(y_n)$ has a Cauchy subsequence. Therefore, the sets $B_1=\{x_n : n \in \N \}$ and $B_2=\{y_n : n \in \N \}$ are disjoint complete sets in $X$. Also, $B_1$ is nowhere locally totally bounded. Therefore, according to the hypothesis $d(B_1, B_2) > 0$, a contradiction.
		
		$(e) \Rightarrow (f)$: Let $G=X\setminus B$. According to the hypothesis, there exists $r > 0$ such that $B(C, r) \subseteq G = G=X\setminus B$. Consequently, $d(B, C) > 0.$
		
		$(f) \Rightarrow (g)$: Suppose $d(B, C)=0$. Then for each $n \in \N$, we can find $x_n \in C$ and $y_n \in B$ such that $d(x_n,y_n) < \frac{1}{n}$. If $(y_n)$ has a Cauchy subsequence $(y_{n_k})$, then since $B$ is complete, there exists $y \in B$ such that $(y_{n_k})$ converges to $y$, hence $(x_{n_k})$ also converges to $y$. Consequently $y \in B \cap C$, which is a contradiction. Hence neither $(x_n)$ or $(y_n)$ has a Cauchy subsequence. Consider the disjoint sets $B_1=\{x_n : n \in \N \}$ and $B_2=\{y_n : n \in \N \}$. Clearly, $B_1$ is complete and almost nowhere locally totally bounded and $B_2$ is closed. Therefore, according to the hypothesis $d(B_1, B_2) > 0$, a contradiction.
		
		$(g) \Rightarrow (h)$: 
		By passing to a subsequence, we can assume $\underline{t}(F_n) < \frac{1}{n}$ and hence for each $n \in \N$, there exists $x_n \in F_n$ with $t(x_n)< \frac{1}{n}$. If $(x_n)$ has a Cauchy subsequence say $(x_{n_k})$, then the subsequence must converge and the convergent point must lie in $\bigcap \{F_n: n\in\mathbb{N}\}$, as $\langle F_n\rangle$ is a decreasing sequence of complete sets. So assume $(x_n)$ has no Cauchy subsequence. Without loss of generality, we can assume $(x_n)$ has all distinct terms. Let $C = \{x_n: n \in \N\}$. Then $C$ is an almost nowhere locally totally bounded as for each $\epsilon >0$, the set $\{x \in (x_n): t(x) \geq \epsilon \}$ is finite. Thus $C$ is not totally bounded and hence we can find a sequence $(a_n)$ in $C$ and $\delta >0$ such that the family of balls $\{B(a_n, \delta): n \in \N\}$ is pairwise disjoint. Since $\lim\limits_{n \to \infty}t(a_n)=0$, by passing to a tail of the sequence we can assume that for each $n$, $~\exists~ e_n \in B(a_n, \delta)\setminus \{a_n\}$ with $\lim\limits_{n \to \infty}d(a_n, e_n)=0$. Consequently, $\lim\limits_{n \to \infty}t(e_n)=0$ and hence the set $E= \{e_n: n \in \N\}$ is complete. But $d(A, E)=0$, where $A=\{a_n: n \in \N\}$ is closed almost nowhere locally totally bounded, a contradiction to the hypothesis.
		
		$(h) \Rightarrow (i)$: Let $F_n$ be a decreasing sequence of non-empty closed subsets of $X$ such that for some m $\in \mathbb{N}$, $F_m$ is complete and $\lim\limits_{n \to \infty} \underline{t}(F_n) = 0$. Then for each $k \in \N,~E_k=F_k \cap F_m$ is complete. Consequently, $E_1 \supseteq E_2 \supseteq ...$ is a decreasing sequence of non empty complete subsets of $X$. Also since $\hat{t}$ is monotone, $\hat{E_n} \leq \hat{F_n}~ \forall n \in \N$ and hence 
		$\lim\limits_{n \to \infty} \underline{t}(E_n) = 0$. So $\bigcap \{E_n: n\in\mathbb{N}\} \neq \phi$ and thus $\bigcap \{F_n: n\in\mathbb{N}\} \neq \phi$.
		
		$(i) \Rightarrow (j)$: This is immediate.
		
		$(j) \Rightarrow (k)$: This is immediate.
		
		$(k) \Rightarrow (a)$: Suppose $(x_n)$ be a sequence in $X$ with $\lim\limits_{n \to \infty}t(x_n)=0$ such that it has no Cauchy subsequence.	Without loss of generality, we may assume that for each $n$, $ t(x_n) \geq t(x_{n+1})$. The set $F_n = {\{x_k: k \geq n\}}$ is complete for each $n \in \N$. Also, $t(F_n) = t(\{x_k: k \geq n\})= t(x_n)$ and hence by $(k)$, $\bigcap \{F_n: n\in\mathbb{N}\}$ is non-empty which is precisely the set of cluster points of $(x_n)$. Since $(x_n)$ has no Cauchy subsequence, we arrive at a contradiction.
		
		$(a) \Rightarrow (l)$: Let $(A_n)$ be a sequence of closed sets in $X$ such that for some $k\in \mathbb{N}$, $A_k$ is complete in $X$ and for some $t\in \mathbb{N},~ (k\neq t)$, $A_t$ is almost nowhere locally totally bounded and $\bigcap\limits_{n=1}^{\infty}A_n=\emptyset$. Suppose $\bigcap\limits_{n=1}^{\infty}B(A_n,1/m) \neq \emptyset~\forall m \in \N$. Let $x_m \in \bigcap\limits_{n=1}^{\infty}B(A_n,1/m)~\forall m \in \N$. In particular, $x_m \in B(A_k, 1/m) ~\forall m \in \N$. Hence for each $m \in \N$, we can find $y_m \in A_k$ such that $d(x_m, y_m) < 1/m$. \\
		Now if $(x_n)$ has a Cauchy subsequence then so has the sequence $(y_n)$, say $(y_{n_k})$. Since $(y_m) \in A_k$ and $A_k$ is complete, the sequence $(y_{n_k})$ converges to some point $x \in A_k$. Hence the sequence $(x_{n_k})$ also converges to the point $x$. Note that $d(x, A_m) \leq d(x, x_{n_k})~+ ~d(x_{n_k}, A_m)$. Since $x_{n_k} \in B(A_m, 1/n_k),~d(x_{n_k}, A_m) < 1/n_k < 1/k$ and consequently $\lim\limits_{k \to \infty}d(x_{n_k}, A_m)=0$. Hence $d(x, A_m)=0$, since $A_m$ is closed, $x \in A_m$ and this is true for all $m \in \N$. Thus, $x \in \bigcap\limits_{n=1}^{\infty}A_n$, which is a contradiction. Hence the sequence $(x_m)$ can not have a Cauchy subsequence.\\
		As chosen previously, choose $(y_m) \in A_t$. Since $A_t$ is almost nowhere locally totally bounded and $(y_m)$ does not have any Cauchy subsequence, we can find a subsequence $(y_{m_k})$ of $(y_m)$ such that $\lim\limits_{k \to \infty} t(y_{m_k}) = 0$. Since the sequence $(y_{m_k})$ does not have any Cauchy subsequence, by Theorem [1] we get contradiction.
		
		$(l) \Rightarrow (m)$: This is immediate.
		
		$(m) \Rightarrow (d)$: This is immediate.
		
		Similarly, we can show that $(a) \Rightarrow (n) \Rightarrow (m) \Rightarrow (d)$.
	\end{proof}		
	\begin{theorem}
		Let $(X,d)$ be a metric space. The following statements are equivalent:
		\begin{enumerate}[label=(\alph*)]
			\item The completion $(\widehat{X},d)$ of $(X,d)$ is cofinally complete.
			\item Whenever $A$ is a complete subset of $X$ and $\{V_i : i \in I\}$ is a collection of open subsets of $X$ with $A\subseteq \bigcup V_i$, then $\exists\thinspace \delta>0$ such that each $x \in X$, there exists finite sets $V_1, V_2,...V_n$ such that $B(x,\delta) \subseteq \{X\setminus A\}~\bigcup~\bigcup\limits_{i=1}^{n}V_i$.
			\item Whenever $A$ is a complete subset of $X$ and $\{V_n : n \in \N\}$ is a countable collection of open subsets of $X$ with $A\subseteq \bigcup V_i$, then $\exists\thinspace \delta>0$ such that each $x \in X$, there exists finite sets $V_1, V_2,...V_n$ such that $B(x,\delta) \subseteq \{X\setminus A\}~\bigcup ~\bigcup\limits_{i=1}^{n}V_i$.
		\end{enumerate}
	\end{theorem}
	\begin{proof}
		$(a) \Rightarrow (b)$: Since $A$ is a complete subset of $X$, $A$ is closed in the completion $(\widehat{X},d)$ of $(X,d)$. For each $i \in I$, $V_i = W_i \cap X$ where $W_i$ is an open subset of $(\widehat{X},d)$. The family $\mathcal{G} = \{W_i : i \in I\} \cap \{\widehat{X}\setminus A \}$ is an open cover of $\widehat{X}$. By [...], $\exists \delta > 0$ such that $\forall x \in X,~ B(x,\delta) = \{y \in \widehat{X} : d(x,y) < \delta \}$ is contained in some finite sets of the cover $\mathcal{G}$. Hence, $\exists \delta > 0$ such that $\forall x \in X$, there exists finite sets $V_1, V_2,...V_n$ such that $B(x,\delta) \subseteq \{X\setminus A\}~\bigcup~\bigcup\limits_{i=1}^{n}V_i$.
		
		$(b) \Rightarrow (c)$: This is immediate.
		
		$(c) \Rightarrow (a)$: We will prove that every complete subset $A$ of $X$ is cofinally complete. Let $A$ be any complete subset of $X$. Let $A \subseteq \bigcup\limits_{i \in I}U_i$, where $U_i$ is an open set in $A$ for each $i$. Thus, for each $i, ~U_i = V_i \cap  A$, for some open set $V_i$ in $X$. Therefore, $A \subseteq \bigcup\limits_{i \in I}(V_i \cap A) \subseteq \bigcup\limits_{i \in I}V_i$. According to the hypothesis, $\exists\thinspace \delta>0$ such that each $x \in X$, there exists finite sets $V_1, V_2,...V_n$ such that $B(x,\delta) \subseteq \{X\setminus A\}~\bigcup ~\bigcup\limits_{i=1}^{n}V_i$. Consequently, $B(x,\delta) \cap A \subseteq \bigcup\limits_{i=1}^{n}V_i$ and hence, $B(x,\delta) \cap A \subseteq \bigcup\limits_{i=1}^{n}U_i$. Hence, $A$ is cofinally complete. Thus $(\widehat{X},d)$ is cofinally complete.
	\end{proof}

	\begin{theorem}
		Let $(X,d)$ be a metric space. The following statements are equivalent:
		\begin{enumerate}[label=(\alph*)]
			\item[$(a)$] The completion $(\widehat{X},d)$ of $(X,d)$ is cofinally complete.
		     \item[$(b)$] 	Whenever $(Y,\rho)$ is a metric space and $f: (X,d) \rightarrow (Y,\rho)$ is Cauchy continuous such that for each $\epsilon > 0$, $f$ is uniformly continuous on $\{x \in X : t(x) > \epsilon\}$, then $f$ is uniformly continuous on $X$.
		     \item[$(c)$] If $f$ is a bounded Cauchy continuous real-valued function on $X$ such that for each $\epsilon >0$, $f$ is uniformly continuous on $\{x \in X : t(x) > \epsilon\}$, then $f$ is uniformly continuous on $X$.
		\end{enumerate}
	\end{theorem}
	\begin{proof}
		$(a) \Rightarrow (b)$: Suppose $f$ is uniformly continuous on $\{x \in X : t(x) > \epsilon\}$ for each $\epsilon > 0$, yet $f$ is not globally uniformly continuous. Therefore there exists $\lambda > 0$ such that for each n $\in \N$, $\exists x_n,~y_n \in X$ such that $d(x_n, y_n) < \frac{1}{n}$ but $\rho(f(x_n), f(y_n)) > \lambda$. Thus for each $\epsilon > 0$, eventually $\{x_n, y_n\} \cap \{x \in X : t(x) \leq \epsilon\} \neq \phi$. Since $t$ is uniformly continuous, $\lim\limits_{n \to \infty} t(x_n) = \lim\limits_{n \to \infty} t(y_n) = 0$. By condition (2) of Theorem [1], there exists a Cauchy subsequence $(x_{n_k})$ of $(x_n)$. Thus the sequence $x_{n_1}, y_{n_1}, x_{n_2}, y_{n_2},...$ is a Cauchy sequence but its image is not Cauchy, a contradiction.
		
		$(b) \Rightarrow (c)$: This is immediate.
		
		$(c) \Rightarrow (a)$: Suppose $(\widehat{X},d)$ is not cofinally complete. Therefore, $\exists (x_n) \in X$ such that $\lim\limits_{n \to \infty} t(x_n) = 0$ but it has no Cauchy subsequence. Thus $\exists \delta > 0$ such that by passing to a subsequence the closed balls $C(x_n, \delta)$ are disjoint $\forall n \in \N$. Also, $\lim\limits_{n \to \infty} t(x_n) = 0$ implies that we can assume $t(x_n) <$ min$\{\delta, \frac{1}{n}\}~\forall n \in \N$. Let $\delta_n$ = min$\{\delta, \frac{1}{n}\}$ $\forall n \in \N$ . As $B(x_n, \delta_n)$ is not totally bounded, we can choose $y_n \neq x_n$ such that $d(x_n, y_n) < \delta_n$.  Let $\epsilon_n = d(x_n, y_n)~\forall n$. Define Lipschitz function $f_n$ for each $n$ as follows:
		$$f_n(x)=\left\{ \begin{array}{lll}
		1- \frac{d(x, x_n)}{\epsilon_n}     & : & x\in C(x_n, \epsilon_n)\\
		~~~~~~~	0     & :& otherwise
		\end{array}\right.$$
		Clearly $f_n$ is Lipschitz (choose Lipschitz constant to be any real number greater than $\frac{1}{\epsilon_n}$). Let $f:= \displaystyle{\sum_{n=1}^{\infty}}f_n$, $f$ is a bounded Cauchy continuous function. Let $\epsilon > 0$, since $t$ is uniformly continuous, $\lim\limits_{n \to \infty} t(x_n) = 0$ and $\lim\limits_{n \to \infty} t(\epsilon_n) = 0$, $\exists k \in \N$ such that $\{x : t(x) > \epsilon\}~\bigcap~ \displaystyle{\bigcup_{n=k+1}^{\infty}}C(x_n, \epsilon_n) = \phi$. Thus $f$ restricted to $\{x : t(x) > \epsilon\}$ is Lipschitz continuous but it is not globally uniformly continuous as $\forall n \in \N$ $d(x_n,y_n) < \frac{1}{n}$ yet $f(x_n)-f(y_n)=1$, a contradiction.
	\end{proof}


\begin{thebibliography}{99}
 \bibliographystyle{plain}   
 \bibitem{[AK1]} 	M. Aggarwal and S. Kundu, \emph{Boundedness of the relatives of uniformly continuous functions}, Top. Proc. 49(2017), 105-119.
\smallskip

\bibitem{[AK2]}	M. Aggarwal and S. Kundu, \emph{More about the cofinally complete spaces and the Atsuji spaces}, Houston J. Math, 42(4): 1373-1395, 2016.
\smallskip

\bibitem{[AK3]}	M. Aggarwal and S. Kundu, \emph{More on variants of complete metric spaces}, S. Acta Math. Hungar. (2017) 151: 391
\smallskip

\bibitem{[MA]} M. Atsuji, \emph{Uniform continuity of continuous functions of metric spaces}, Pacific J. Math. 8 (1958), 11-16.
\smallskip

\bibitem{[B1]} G. Beer, \emph{Between compactness and completeness}, Top. Appl. 155 (2008), 503-514.
\smallskip

\bibitem{[B2]} G. Beer, \emph{More about metric spaces on which continuous functions are uniformly continuous}, Bull. Australian Math. Soc. 33 (1986), 397-406.
\smallskip

\bibitem{[BG]} G. Beer and M. I. Garrido, \emph{Bornologies and locally Lipschitz functions}, Bull. Aust. Math. Soc. 90 (2014), 257-263.
\smallskip

\bibitem{BG1} 	G. Beer, M. I. Garrido and Ana S. Merono, \emph{Uniform continuity and a new bornology for a metric space}, Set-Valued Var. Anal (2018) 26:49
\smallskip

\bibitem{BM} G. Beer and G. Di Maio, \emph{The Bornology of cofinally complete subsets}, Acta Math. Hungar., 134(3)(2012), 322-343.
\smallskip

\bibitem{[H2]} N. Howes, On completeness, Pacific J. Math. 38 (1971) 431-440.
\smallskip
\bibitem{J[K2]} T. Jain and S. Kundu, \emph{Atsuji completions: Equivalent characterisations}, Top. Appl. 154 (2007), 28-38.
\smallskip

\bibitem{[JK1]} T. Jain and S. Kundu, \emph{Boundedly UC spaces: characterisations and preservation}, Quaestiones Mathematicae 30(2007), 247-262
\smallskip

\bibitem{[K]} K. Keremedis, \emph{Metric spaces on which continuous functions are almost uniformly continuous}, Top. Appl. 232(2017), 256-266.
\smallskip

\bibitem{[KJ1]} S. Kundu and T. Jain, \emph{Atsuji spaces: Equivalent conditions}, Top. Proc. 30 (2006), 301-325.
\smallskip
\bibitem{[JN]} J. Nagata, \emph{On the uniform topology of bicompactifications}, J. Inst. Polytech, Osaka City University 1 (1950), 28-38.
\smallskip
\bibitem{[R]} M. Rice, \emph{A note on uniform paracompactness}, Proc. Amer. Math. Soc. 62(1977), 359-362.
\smallskip
\bibitem{[S]} R. F. Snipes, \emph{Functions that preserve Cauchy sequences}, Nieuw Arch. Wisk.(3),25(3):409-422, 1977.
\smallskip
\bibitem{[T]} G. Toader, \emph{On a problem of Nagata}, Mathematica 20 (1978) 77–79.



%

\end{thebibliography}
\end{document}